\documentclass[a4paper,reqno,12pt]{amsart}

\usepackage{amsmath, amssymb, amsthm, enumerate,amsfonts,latexsym, mathrsfs}

\usepackage[top=30mm,right=27mm,bottom=30mm,left=27mm]{geometry}

\newtheorem{theorem}{Theorem}[section]

\newtheorem{lemma}[theorem]{Lemma}

\makeatletter
\newcommand{\imod}[1]{\allowbreak\mkern4mu({\operator@font mod}\,\,#1)}
\makeatother

\newcommand{\Irr}{{\mathrm {Irr}}}
\newcommand{\cd}{{\mathrm {cd}}}
\newcommand{\cv}{{\mathrm {cv}}}
\newcommand{\Aut}{{\mathrm {Aut}}}
\newcommand{\Out}{{\mathrm {Out}}}

\newcommand{\PSL}{{\mathrm {PSL}}}
\newcommand{\GL}{{\mathrm {GL}}}

\newcommand{\SL}{{\mathrm {SL}}}

\newcommand{\PGL}{{\mathrm {PGL}}}

\newcommand{\Atlas}{{\sf Atlas}}

\theoremstyle{definition}

\begin{document}
\title[\textbf{Groups with few character values}]{\textbf{Finite groups with few character values}}

\author{Sesuai Y. Madanha}
\address{Department of Mathematics and Applied Mathematics, University of Pretoria, Private Bag X20, Hatfield, Pretoria 0028, South Africa}
\email{sesuai.madanha@up.ac.za}


\date{\today}


\begin{abstract}
A classical theorem on character degrees states that if a finite group has fewer than four character degrees, then the group is solvable. We prove a corresponding result on character values by showing that if a finite group has fewer than eight character values in its character table, then the group is solvable. This confirms a conjecture of T. Sakurai. We also classify non-solvable groups with exactly eight character values.
\end{abstract}

\maketitle


\section{Introduction}\label{s:intro}
Let $ G $ be a finite group and $ \Irr(G) $ be the set of complex irreducible characters of $ G $. Recall the definition of a  character degree set:

\begin{center}
$ \cd(G):=\{\chi(1)\mid \chi \in \Irr(G) \} $.
\end{center} 

The study of the set $ \cd(G) $ has received much attention from many authors and has been shown to have strong influence on the structure of $ G $. In particular, we have the famous Ito-Michler theorem which states that $ G $ has a normal abelian Sylow $ p $-subgroup if and only if $ p $ does not divide any member of $ \cd(G) $. On the other extreme is Thompson's theorem, which states that if $ p $ divides every member of $ \cd(G)\setminus \{ 1\} $, then $ G $ has a normal $ p $-complement. In this note we shall study a bigger set found in a character table:

\begin{center}
$ \cv(G)=\{\chi(g)\mid \chi \in \Irr(G), g\in G \} $.
\end{center}

Little has been done in the study of this set. The first article in this direction is from T. Sakurai \cite{Sak20}, who studied groups with very few character values. We shall compare results in \cite{Sak20} with corresponding character degree results. It is well known that $ |\cd(G)|=1  $ if and only if $ G $ is abelian. Sakurai showed that if $ |\cv(G)|\leq 3 $, then $ G $ is abelian. The converse is not true since $ |\cv(\mathrm{C}_{r})|=r $, where $ \mathrm{C}_{r} $ denote the cyclic group of order $ r $ for some positive integer $ r $.

If $ |\cd(G)|=2 $ with $ \cd(G)=\{ 1, n\} $, then either $ G $ has an abelian normal subgroup of index $ n $, or $ n=p^{e} $ for a prime $ p $ and $ G $ is direct product of a $ p $-group and abelian group. It turns out that a non-abelian group $ G $ is such that $ |\cv(G)|=4 $ if and only if $ G $ is the generalized dihedral group of a non-trivial elementary abelian $ 3 $-group (\cite[Theorem]{Sak20}). A generalized dihedral group of a non-trivial elementary abelian $ p $-group is a group $ G_{n}= [\mathrm{C}_{p}\times \mathrm{C}_{p}\times \cdots \times \mathrm{C}_{p}]\rtimes \mathrm{C}_{2} $, the semidirect product of the direct product of $ n $ copies of $ \mathrm{C}_{p} $ and $ \mathrm{C}_{2} $, where $ \mathrm{C}_{2} $ acts on $ [\mathrm{C}_{n}\times \mathrm{C}_{n}\times \cdots \times \mathrm{C}_{n}] $ without non-trivial fixed points. For groups with more than four character values, Sakurai remarked on the difficulty to classify these groups. However, a conjecture was proposed which states that if the group has fewer than eight character values, then the group is solvable \cite[Remark]{Sak20}. Our aim is to settle this conjecture:
\begin{theorem}\label{thmA}
If $ |\cv(G)|<8 $, then $ G $ is solvable.
\end{theorem}

Isaacs \cite[Theorem 12]{Isa06} proved a result that is now considered classical which states that if  $ |\cd(G)|<4 $, then $ G $ is solvable. Hence Theorem \ref{thmA} is the corresponding result for character values. Our proof uses the classification of finite simple groups.

Malle and Moret\'o \cite{MM05} classified non-solvable groups with $ |\cd(G)|=4 $ (see Theorem \ref{MM05TheoremA}). We obtain a corresponding result for character values:
\begin{theorem}\label{thmB}
Let $ G $ be a finite non-solvable group. Then $ |\cv(G)|=8 $ if and only if $ G\cong \PSL_{2}(5) $ or $ \PGL_{2}(5) $.
\end{theorem} 

\section{Preliminary Results}
In this section we list some preliminary results that we need to prove our main results. We shall use freely these well known results in the following lemma. 
\begin{lemma}\label{lem}
Let $ G $ be a finite group, $ N $ be a normal subgroup of $ G $ and $ n $ be a positive integer. Then the following hold:
\begin{itemize}
\item[(a)] $ |\cv(G/N)|\leq |\cv(G)| $,
\item[(b)] The character table of $ \mathrm{C}_{n} $ has a column with pairwise different values. In particular, $ |\cv(\mathrm{C}_{n})|=n $,
\item[(c)] If $ G $ is non-abelian, then $ 0\in \cv(G) $.
\end{itemize}
\end{lemma}

The following is a result of Malle and Moret\'o \cite{MM05} in which they classified non-solvable groups with four character degrees. 
\begin{theorem}\cite[Theorem A]{MM05}\label{MM05TheoremA}
Let $ G $ be a non-solvable group with $ |\cd(G)|=4 $. Then one of the following holds:
\begin{itemize}
\item[(a)] $ G\cong \PSL_{2}(2^{f})\times A $ for some $ f\geq 2 $ and some abelian subgroup $ A $;
\item[(b)] $ G $ has a normal subgroup $ U $ such that $ U\cong \PSL_{2}(q) $ or $ \SL_{2}(q) $ for some odd $ q\geq 5 $, and if $ C=\textbf{C}_{G}(U) $, then $ C\leq \textbf{Z}(G) $ and $ G/C\cong \PGL_{2}(q) $; or
\item[(c)] the group $ G $ has a normal subgroup of index $ 2 $ that is a direct product of $ \PSL_{2}(9) $ and central subgroup $ C $. Furthermore, $ G/C\cong M_{10} $, the Mathieu group which the stabilizer of a point in the Mathieu group, $ M_{11} $, in its natural permutation representation.
\end{itemize}
Conversely, if one of (a)-(c) holds, then $ |\cd(G)|=4 $.
\end{theorem}

He and Zhu \cite{HZ12} described non-solvable groups with five and six character degrees:

\begin{theorem}\cite[Corollary C]{HZ12}\label{cd5,6}
Let $ G $ be a non-solvable group and $ L $ be the solvable radical of $ G $. If $ 5\leq |\cd(G)|\leq 6 $, then $ G'/L $ is isomorphic to  $ \PSL_{2}(p^{f}) $ with $ p^{f}\geq 4 $, $ \PSL_{3}(4) $ or $ ^{2}\mathrm{B}_{2}(2^{f}) $, where $ 2^{f}=2^{2m + 1} $ for some integer $ m\geq 1 $. Moreover, $ G/L $ is an almost simple group.
\end{theorem}

We classify almost simple groups with five character degrees. We first have some definitions needed in our result : $ q=p^{f}\geqslant 4 $, $ \Out(\PSL_{2}(q))=\langle \delta \rangle \times \langle \varphi \rangle , |\langle \varphi \rangle|=f, |\langle \delta \rangle|=\gcd(2, q-1) $.
\begin{theorem}\label{cd5}
Let $G$ be a finite almost simple group, that is $ S \triangleleft G \leq Aut(S) $ for some non-abelian simple group $S$. Then $ |\cd(G)|=5 $ if and only if one of the following holds:
\begin{itemize}
\item[(a)]  $ G\cong \PSL_{2}(q) $, $ q $ odd, $ \cd(G)=\{1, (q+\epsilon)/2, q \pm 1, q\} $,
\item[(b)] $ G\cong\PSL_{2}(2^{f})\langle \varphi \rangle $, $ f > 2 $ prime, $ \cd(G)=\{1, 2^{f}-1, 2^{f}, (2^{f}\pm 1)f\} $,
\item[(c)] $ G\cong \PSL_{2}(2^{f})\langle \varphi^{f/2} \rangle $, $ f > 2 $ even, $ \cd(G)= \{1, 2^{f}, 2^{f} + 1, 2(2^{f}\pm 1)\} $,
\item[(d)] $ G\cong \PGL_{2}(q)\langle \varphi^{f/2} \rangle $, $ f > 2 $ even, $ q $ odd, $ \cd(G)=\{1, q, q + 1, 2(q\pm 1) \} $,
\item[(e)] $ G\cong \PGL_{2}(3^{f})\langle \varphi \rangle $, $ f > 2 $ even, $ \cd(G)=\{1, 3^{f}, 3^{f} - 1, (3^{f}\pm 1)f\} $,
\item[(f)] $ G\cong \PSL_{2}(9)\langle \varphi \rangle \cong \mathrm{S}_{6} $, $ f = 2 $, $ \cd(G)=\{ 1, 5, 9, 10, 16\} $,
\item[(g)] $ G\cong \PSL_{2}(5^{f})\langle \varphi \rangle $, $ f > 2 $ prime, $ \cd(G)=\{1, 5^{f}, (5^{f} + 1)/2, 5^{f} - 1, (5^{f}\pm 1)f \} $,
\item[(h)] $ G\cong \PSL_{2}(q)\langle \delta \varphi ^{f/2}\rangle $, $ f > 2 $ even, $ q\neq 9 $, $ p $ odd, $ \cd(G)=\{1, q, q + 1, 2(q\pm 1) \} $.
\end{itemize}
\end{theorem}
\begin{proof} 
First, it is well known that every simple group $ S $ has a non-trivial irreducible character $\chi $ that can be extended to $ \Aut(S) $. Suppose that $ r $ is a prime that divides the character degree of $ \chi $. If $ r $ divides every character degree of $ G $, then by Thompson's theorem \cite[Corollary 12.2]{Isa06}, $ G $ has a normal $ r $-complement, a contradiction. 

We may thus assume that no prime divides every character degree of $ G $. Since $ G $ has five character degrees, there is no prime dividing four character degrees of $ G $. These groups are the groups in \cite[Theorem A]{GGLTV17}. It is sufficient to consider groups in case (c) since the groups in (a)-(b) and (d)-(f) can be ruled out using the character tables in the \Atlas{} \cite{CCNPW85} and \cite{Suz62}.

By inspecting \cite[Table 1]{GGLTV17}, we have our result.
\end{proof}

\section{Main Results}

\begin{theorem}\label{almosthasgreaterthan8}
Let $G$ be a finite almost simple group, that is $ S \triangleleft G \leq Aut(S) $ for some non-abelian simple group $S$. Then either $ |\cv(G)|\geq 9 $ or $ |\cv(G)|=8 $ and $ G\in \{\PSL_{2}(5), \PGL_{2}(5)\} $. In particular, if $ |\cd(G)|\geq 5 $, then $ |\cv(G)|\geq 9 $.
\end{theorem}
\begin{proof}
Since $ G $ is non-solvable, $ |\cd(G)|\geq 4 $ by \cite[Theorem 12.15]{Isa06}. Suppose that $ |\cd(G)|=4 $. By Theorem \ref{MM05TheoremA}, it is sufficient to consider the following groups: $ \PSL_{2}(2^{f}) $ for some $ f\geq 2 $, $ \PGL_{2}(q) $ for some odd $ q\geq 5 $ and $ M_{10} $. The character tables of these groups can be found in the \Atlas{} \cite{CCNPW85}, \cite{Geh02} and \cite{Ste51}. In particular, $ |\cv(G)|\geq 9 $ except for $ \PSL_{2}(5) $ and $ \PGL_{2}(5) $. In the latter cases $ |\cv(G)|=8 $ as required.

We may assume that $ |\cd(G)|=5 $. It is sufficient to consider groups in Theorem \ref{cd5}. For cases (a) and (f), we have that $ |\cv(G)|\geq 9 $, using character tables in \cite{Geh02} and \Atlas{} \cite{CCNPW85}. For cases (b), (e) and (g) we have that $ C=G/\PGL_{2}(q) $ is cyclic of order greater than $ 2 $. Then by Lemma \ref{lem}(b), the character table of $ C $ has a column with pairwise different values. In particular, $ a,b\in \cv(C)\setminus \{-1,0,1\} $ with $ a\not= b $. Note that the Steinberg character $ \psi $ of $ \PSL_{2}(q) $ is extendible to $ G $ and also $ \psi(s)=-1 $ for some $ s\in \PSL_{2}(q) $ by \cite[Theorems 4.7 and 4.9]{Geh02}. Hence $ -1\in \cv(G) $ and $ |\cv(G)|=|\cd(G)\cup \{a, b, -1, 0 \}|\geq 9 $. Suppose that $ C $ is of order $ 2 $. Groups with this property are the ones in cases (c), (d) and (h). Consider the characters $ \chi_{1}, \chi_{2}\in \Irr(G) $ such that $ \chi_{1}(1)=2(q + 1) $ and $ \chi_{2}(1)=2(q - 1) $. Then since $ \PGL_{2}(q) $ has no irreducible characters of degrees $ 2(q\pm 1) $, $ \chi_{i}=\theta_{i1} + \theta_{i2} $ by \cite[Theorem 6.2]{Isa06}, where $ \theta_{11}, \theta_{12}, \theta_{21}, \theta_{22}\in \Irr(S) $, $ \theta_{11}(1)=\theta_{12}(1)=q+1 $ and $ \theta_{21}(1)=\theta_{22}(1)=q-1 $. Using \cite[Theorems 4.7 and 4.9]{Geh02}, we have that $ \theta_{11}(g)=\theta_{12}(g)=1 $ and $ \theta_{21}(g)=\theta_{22}(g)=-1 $ for some $ g\in \PGL_{2}(q) $. Hence $ \chi_{1}(g)=2 $, $ \chi_{1}(g)=-2 $ and $ |\cd(G)\cup \{-2, -1, 0, 2\}|\geq 9 $.

Going forward we may assume that $ |\cd(G)|\geq 6 $. Suppose that $ |\cd(G)|\geq 6 $ and $ S\cong \PSL_{2}(q) $. We may assume that $ G $ is not isomorphic to $ \PSL_{2}(2^{f}) $ with $ f\geq 2 $, $ \PSL_{2}(p^{f}) $ with $ p $ odd or $ \PGL_{2}(p^{f}) $ with $ p $ odd. Let $ H\in \{ \PSL_{2}(2^{f}), f\geq 2\} \cup \{\PSL_{2}(q) $, $ q $ odd $\} \cup \{\PGL_{2}(q), q $ odd$ \} $. If $ G\cong H\mathrm{C}_{r} $, where $ r\mid f $, $ r > 2 $. Then arguing as above, we have that $ a,b\in \cv(\mathrm{C}_{r})\setminus \{-1, 1\} $ and $ |\cv(G)|=|\cd(G)\cup \{-1, 0, a, b \}|\geq 10 $. If $ G\cong H\mathrm{C}_{2} $, $ 2\mid f $. Then either $ 2(q-1) $ or $ 2(q+1) $ is a character degree of $ G $. Arguing as above, we have that either $ -2\in \cv(\mathrm{C}) $ or $ 2\in \cv(\mathrm{C}) $. Since $ -1, 0\in \cv(G) $ and $ |\cd(G)|\geq 6 $ we have that $ |\cv(G)|\geq 9 $.

If $ S\cong \PSL_{3}(4) $, then the result follows using the character table in the \Atlas{} \cite{CCNPW85}. Suppose that $ S\cong ^{2}\!\!\mathrm{B}_{2}(2^{f}) $, where $ 2^{f}=2^{2m + 1} $ for some integer $ m\geq 1 $. If $ G\cong ^{2}\!\!\mathrm{B}_{2}(2^{f}) $, it follows that $ |\cv(G)|\geq 9 $ by the character table in \cite[Theorem 13]{Suz62}. We may assume that $ S < G < \Aut(S) $. Then $ C=G/S $ is cyclic of odd order. Using the argument as in the case of $ \PSL_{2}(q) $ above, we have that $ a,b\in \cv(C)\setminus \{-1, 1\} $. Hence $ |\cv(G)|\geq |\cd(G)\cup \{a,b,0\}|=9 $. 

If $ |\cd(G)|\geq 7 $, then by column orthogonality, there is a character value which is a negative number and so $ |\cv(G)|\geq 9 $. This concludes our proof.
\end{proof}

\begin{proof}[\textbf{Proof of Theorem \ref{thmA}}] We prove our result using induction on $ |G| $. Since $ |\cv(G)|<8 $, every homomorphic image of $ G $ has the same property. Let $ N $ be a minimal normal subgroup of $ G $. If another such minimal normal subgroup, say $ M $, exists then $ G $ is the subdirect product of solvable groups $ G/N $ and $ G/M $, and so is solvable. This implies that $ N $ is unique and non-abelian. Hence $ G $ is almost simple group. The result follows by Theorem \ref{almosthasgreaterthan8}.
\end{proof}
Before we prove Theorem \ref{thmB}, we recall a definition. If $ N\triangleleft G $, then $ \Irr(G{\mid} N)=\{\chi \in \Irr(G)\mid N\not \subseteq \ker \chi \} $ and $ \cd(G{\mid}N)=\{\chi(1)\mid \chi \in \Irr(G{\mid}N)\} $.
\begin{proof}[\textbf{Proof of Theorem \ref{thmB}}]
Since $ G $ is non-solvable, $ |\cd(G)|\geq 4 $. If $ 5\leq |\cd(G)|\leq 6 $, then by Theorem \ref{cd5,6}, $ G'/L $ is isomorphic to  $ \PSL_{2}(p^{f}) $, $ p^{f}\geq 4 $, $ \PSL_{3}(4) $ or $ ^{2}\mathrm{B}_{2}(2^{f}) $, where $ 2^{f}=2^{2m + 1} $ for some integer $ m\geq 1 $, where $ L $ is the solvable radical of $ G $. Specifically, $ G/L $ is an almost simple group. Using Theorem \ref{almosthasgreaterthan8}, $ |\cv(G)|\geq|\cv(G/L)|\geq 9 $. 

If $ |\cd(G)|=4 $, then using Theorem \ref{MM05TheoremA} and Theorem \ref{almosthasgreaterthan8}, it is sufficient to consider $ G\cong \PSL_{2}(5)\times A $ for some abelian subgroup $ A $ or $ G/C\cong \PGL_{2}(5) $, where $ C=\textbf{C}_{G}(U) $, $ U=\PSL_{2}(5) $ or $ \SL_{2}(5) $ with $ C\leq \textbf{Z}(G) $ and $ U $ is a normal subgroup of $ G $. Suppose that $ G\cong \PSL_{2}(5)\times A $ with non-trivial $ A $. Consider $ \theta \in \Irr(\PSL_{2}(5)) $ such that $ \theta(1)=5 $ and $\phi \in \Irr(A) $ such that $ \phi(g)=k < 0 $. Then $ \chi(1\times g)=\theta(1)\times \phi(g)=5k<0 $ and $ |\cv(G)|\geq 9 $. Suppose that $ G $ is a group in the latter case. We want to show that $ C=1 $. Suppose that $ C\neq 1 $. Let $ \textbf{Z}(G)U=N $ and note that $ \cd(G{\mid}N)\geq 3 $ by \cite[Theorem B]{IK98}. Then there is a $ \chi \in \Irr(G) $ such that $ \chi(1)=m\in \cd(G{\mid}N)\setminus \{ 1, 2\} $. For $ 1\neq z\in \textbf{Z}(G) $, we have that $ \chi(z)=-m $. This means that $ |\cv(G)|\geq |\PGL_{2}(5)\cup \{-m \}|=9 $. Therefore $ G=\PGL_{2}(5) $ and that concludes our argument.
\end{proof}

\section*{Acknowledgements}
The author would like to thank the reviewer for their comments which improved the presentation of the article.


\begin{thebibliography}{1}
\bibitem{CCNPW85} J. H. Conway, R. T. Curtis, S. P. Norton, R. A. Parker and R. A. Wilson, \textit{Atlas of Finite Groups}, Clarendon Press, 1985.
\bibitem{GGLTV17} M. Ghaffarzadeh, M. Ghasemi, M. L. Lewis, H. P. Tong-Viet, Nonsolvable groups with no prime dividing four
character degrees, \emph{Algebr. Represent. Theory} \textbf{20} (2017) 547--567.
\bibitem{Geh02} K. E. Gehles, Ordinary characters of finite special linear groups, MSc thesis, University of St Andrews, 2002.
\bibitem{HZ12} L. He and G. Zhu, Nonsolvable normal subgroups and irreducible character degrees, \emph{J. Algebra} \textbf{372} (2012) 68--84.
\bibitem{Isa06} I. M. Isaacs, \textit{Character Theory of Finite Groups}, Amer. Math. Soc., Providence, Rhode Island, 2006.
\bibitem{IK98} I. M. Isaacs, G. Knutson, Irreducible character degrees and normal subgroups, \emph{J. Algebra} \textbf{199} (1998) 302--326.
\bibitem{MM05} G. Malle and A. Moret\'o, Nonsolvable groups with few character degrees, \emph{J. Algebra} \textbf{294} (2005) 117--126.
\bibitem{Sak20} T. Sakurai, Finite groups with very few character values, \emph{Comm. Algebra} \textbf{49} (2021) 658--661.
\bibitem{Ste51} R. Steinberg, The representations of $ \GL(3, q)$, $\GL(4, q)$, $\PGL(3, q)$ and $ \PGL(4, q)$, \emph{Canad. J. Math.} \textbf{3} (1951) 225--235.
\bibitem{Suz62} M. Suzuki, On a class of doubly transitive groups, \emph{Ann. of Math.} \textbf{75} (1962) 105--145.
\end{thebibliography}
\end{document}